\documentclass[twoside,11pt]{article}%
\usepackage{amssymb}
\usepackage{bm}
\usepackage{latexsym}
\usepackage{amsmath,amssymb,epsfig,subfigure}
\usepackage{amsthm,enumerate,verbatim}
\usepackage{amsfonts}
\usepackage{color}
\usepackage{amsmath}
\usepackage{graphicx}
\usepackage{mathrsfs}
\usepackage{float}
\providecommand{\U}[1]{\protect\rule{.1in}{.1in}}
\providecommand{\U}[1]{\protect\rule{.1in}{.1in}}
\newcommand{\diag}{{\rm diag}}

\newcommand{\BE}{\begin{equation}}
\newcommand{\EE}{\end{equation}}

\numberwithin{equation}{section}
\newtheorem{proposition}{Proposition}[section]
\newtheorem{theorem}[proposition]{Theorem}
\newtheorem{lemma}[proposition]{Lemma}

\newtheorem{definition}[proposition]{Definition}
\newtheorem{remark}[proposition]{Remark}

\begin{document}

	\title{{\bf {Exponential stability  for {time-delay} neural networks via {new}
			weighted integral inequalities}}}
	\author{
		Seakweng Vong\thanks{
			Email: swvong@umac.mo. This research is funded by The Science and Technology Development Fund, Macau SAR (File no. 0005/2019/A) and funded by University of Macau (File no. MYRG2018-00047-FST, MYRG2017-00098-FST).}
		\qquad
		 Kachon Hoi
		\thanks{Email:
			db62522@umac.mo.}\qquad
		Chenyang Shi\thanks{Corresponding author. Email: shichenyang15@gmail.com.}
		\\
		{\footnotesize  \textit{Department of Mathematics, University of Macau, Avenida da Universidade, Macau, China}}}
	\date{}
	\maketitle	
	
	\begin{abstract}
		We study  exponential stability for a kind of neural networks having time-varying delay.
		{By extending the auxiliary function-based integral inequality,
			 a novel integral inequality is derived by using weighted orthogonal functions of which one is discontinuous. Then, the new inequality is applied to investigate the exponential stability of time-delay neural networks via Lyapunov-Krasovskii functional (LKF) method. Numerical examples are given to verify the advantages of the proposed criterion.}
	\end{abstract}

	\section{Introduction}
With the development of new scientific technology, neural networks have been adopted to different applications such as image decryption, pattern recognition, and finance \cite{XiaoMQ1,LiuY1,ShenH,YanH,ZhangL}.
In general, a practical neural network involves a lot of  neurons to perform several complex tasks.
As it was particularly pointed out in \cite{MarcusC}, time delays of information exchange between these large number of neurons are  unavoidable.
At the same time, { a delay term in a system}, even though it may not be large, is usually a key factor to cause  a neural network unstable.
Therefore, the problem of studying the effect of time delay plays a critical role  in the study of  neural networks
and this issue has been extensively studied in recent years\cite{LiuY2,LiuY3}.

{ As we know, it is an important task to make stability criteria  less conservative when analyzing
time-delay systems \cite{XiaoMQ3,WangS1,XiaoMQ4,KLTeo2,KLTeo3}}.
More specifically, exponential stability is desirable for some applications\cite{XiaoMQ2,Hien,Vong2,Vong3}.
To this end, various approaches have been developed to the subject,
which include techniques of free weighting matrices \cite{HeYLiu}, reciprocally convex {optimization} \cite{ParkP1} and delay-partitioning \cite{GuK,WangZ}.
Exponential stability analysis of time-delay neural networks aims at deriving an admissible delay upper bound (ADUB)  such that the delayed neural networks are stable for all time-delays less than the obtained ADUB.
Roughly speaking, ADUB measures the conservatism of a stability criterion.
If a stability criterion can yield a larger or sharper ADUB than another one, the criterion is less conservative.
It is shown in \cite{HeY} that the LKF method plus linear matrix inequality (LMI) technique is useful to determining the ADUB for delayed systems.

Until now, a great number of integral inequalities have been
	proposed to study delayed systems, such as the Wirtinger-based inequality \cite{Vong2,Vong3},
	the Bessel-Legendre inequality \cite{LIZ2,LIZ1,LIZ3} and the auxiliary function-based
	inequalities \cite{LIZ4,LiuX}. In this paper, we study the auxiliary function-based inequality which was developed for the theoretical study of the delayed systems in \cite{ParkP}.
Since it gives tighter estimates than the Jensen inequality,  compared with the extend form of Jensen inequality \cite{Hien2,Levan1},
employing this inequality can yield  less conservative asymptotic stability criteria for some delayed systems.
Based on \cite{ParkP}, a further improved integral inequality was established in \cite{LiuX} by considering
a group of orthogonal functions of which one is discontinuous.
Instead of using high-degree polynomial to sharpen the bound, a discontinuous function  is employed to reduce the
{number of decision variables (NODVs)}.

In the current study, we aim to investigate the exponential stability of neural networks having time-varying delay
following ideas in \cite{ParkP,LiuX}. Different from \cite{LiuX},
basing on our previous study  \cite{Vong}, {we considered the
orthogonal sets  $\{p_0({\cdot}),\ \ p_1(\cdot),\ \ p_2(\cdot)\}$ with respect to an exponential term, which is called a weight function.}
By combining the decomposition of the state vectors,
{which consists of polynomials,} and
$\{p_0(\cdot),\ \ p_2(\cdot),\ \ p_3(\cdot)\}$,
where $p_3(\cdot)$ is a discontinuous function,
a novel weighted inequality is established.
{Our new weighted inequality was derived by improving the one in \cite{LiuX}
and estimating integrals with exponential term as a whole.} 
An improved  criterion which guarantees exponential stability of neural networks having time-varying delay
is {derived by using our new inequality.
Numerical examples are given to confirm the
advantage of the method proposed in this paper.}

The following points outline the main contributions of this paper:
\begin{itemize}
\item[1.]  A new weighted inequality, {which generalizes the integral inequality based on auxiliary function in \cite{ParkP}, is established.}
The inequality can be used to establish an improved exponential stability criterion for  delayed  systems.

\item[2.]
We consider time-varying delay which is not necessary   nondecreasing (noting that nondecreasing was assumed in the proof of     \cite{LiuX}).
By further studying  the LKF in \cite{LiuX}, we find that some terms in the LKF can be removed to reduce the number of decision variables  in the stability criterion without affecting its performance.

\end{itemize}

{We  organize our paper  as follows. The model of neural networks with time-varying delay and the new weighted inequality are introduced in Section 2. By using a refined LKF,  our main theoretial result is given in Section 3. For the last section, simulations are carried out to demonstrate the proposed criterion. }
	

\bigskip
{\bf Notations:}
{We use $\mathbb{R}^n$ and $\mathbb{R}^{n\times m}$ are the sets of $m$-dimensional Euclidean vector space and $n\times m$ real matrix space.}
When a real matrix $P$ is symmetric and positive definite (semi-positive definite), we describe this using $P>0~ (\geq0)$.
The notation $\diag\{\cdots\}$ refers to  a diagonal matrix. Additionally, we take $\mathbb{S}_n^{+}$ as the set of symmetric positive definite matrices and symmetric terms in a symmetric matrix are marked as $*$ for simplicity of presentation.
Finally, we define $sym(A)=A+A^T$, where   $T$ represents transpose of a matrix

\section{Preliminaries}
{We study time-delay neural networks as follow }
\begin{equation} \label{a1}
\dot{x}(t)=-Cx(t)+Ag(x(t))+Bg(x(t-h(t)))+u,
\end{equation}
where the neuron state vector is denoted by $x(\cdot)=[x_{1}(\cdot),x_{2}(\cdot),\ldots,x_{n}(\cdot)]^{T}\in \mathbb{R}^{n}$ and the activation function is $g(x(\cdot))=[g_{1}(x_{1}(\cdot)),g_{2}(x_{2}(\cdot)),\ldots,g_{n}(x_{n}(\cdot))]^{T}\in \mathbb{R}^{n}$
. The vector {$u=[u_1,u_2,\ldots,u_{n}]^{T}\in \mathbb{R}^{n}$} is an input to the network.
Entries of the matrix $C=\emph{diag}\{c_{1},c_{2},\ldots,c_{n}\}$ satisfy $c_{i}>0$.
{The matrices $A$ and $B$ are weight matrices corresponding to connection.
The differentiable function $h(t)$ denotes the  time-varying delay} and it holds that
\begin{equation} \label{a2}
0\leq h(t)\leq h
\end{equation}
and
\begin{equation} \label{a3}
{|\dot{h}(t)|}\leq \mu
\end{equation}
for some constants $\mu$ and $h$.
As in previous studies,
we  assumed that each activation function of (\ref{a1})
satisfies:
\begin{equation} \label{a4}
0\leq \frac{g_{j}(x)-g_{i}(y)}{x-y}\leq L_{j},\quad x,y\in \mathbb{R},\quad x\neq y,\quad j=1,2,\ldots,n,
\end{equation}
for some positive constants $L_{j},~j=1,2,\ldots,n$.\\*
\indent Under (\ref{a4}),
there exists an $x^{*}=[x^{*}_{1},x^{*}_{2},\ldots,x^{*}_{n}]^{T}$ such that
\begin{equation} \label{a5}
Cx^{*}=Ag(x^{*})+Bg(x^{*})+u.
\end{equation}
\indent
Then shift the the equilibrium point $x^{*}$ of system (\ref{a1}) to the origin by the transform
$z(\cdot)=x(\cdot)-x^{*}$. Then $z=[z_{1}(\cdot),z_{2}(\cdot),\ldots,z_{n}(\cdot)]^{T}$ satisfies
\begin{equation} \label{a6}
\dot{z}(t)=-Cz(t)+Af(z(t))+Bf(z(t-h(t)))
\end{equation}
where $f(z(\cdot))=[f_{1}(z_{1}(\cdot)),f_{2}(z_{2}(\cdot)),\ldots,f_{n}(z_{n}(\cdot))]^{T}$
and $f_{j}(z_{j}(\cdot))=g_{j}(z_{j}(\cdot)+x^{*}_{j})-g_{j}(x_{j}^{*}), j=1,2,\ldots,n$.
With these notations, we have
\begin{equation} \label{a7}
0\leq \frac{f_j(z_{j})}{z_{j}}\leq L_{j},\quad f_{j}(0)=0,\quad \forall z_{j}\neq 0,\quad j=1,2,\ldots,n.
\end{equation}
Definition of exponential stability of  (\ref{a6}) is given below.
\begin{definition} \label{Definition1}
{\rm\cite{LiuX}} { The neural network (\ref{a6}) is exponentially stable at the origin if, for $t>0$,
$$\|z(t)\|\,\leq H\phi e^{-kt}$$
holds for some positive constants $k>0$ and $H\geq1$,
where $\phi =sup_{-h\leq \theta\leq 0}\|z(\theta)\|\,$. In this situation, we call $k$  the exponential convergence rate.}\\
\end{definition}
{The well-known reciprocally convex inequality are useful for the theoretical proof and it is summarized as below:}
\begin{lemma} \label{Lemma1}
{\rm\cite{Park.P}} Suppose that $f_1,f_2,\ldots,f_n:\mathbb{R}^{m}\rightarrow \mathbb{R}$ take positive values in an open subsets $D$ of $\mathbb{R}^{m}$ then the below equation holds:
\begin{equation} \label{a8}
\min\limits_{\Big\{\alpha_i\mid\alpha_i>0,\sum\limits_{i} a_i=1\Big\}}\sum\limits_{i} \frac{1}{\alpha_i}f_{i}(t)=\sum\limits_{i} f_i(t)+\max_{g_{ij}(t)}\sum\limits_{i\neq j} g_{ij}(t)
\end{equation}
subject to
\begin{equation} \label{a9}
\Bigg\{g_{ij}:\mathbb{R}^{m}\rightarrow \mathbb{R},g_{ji}\triangleq g_{ij},
\left[ \begin{array}{ccc}
f_{i}(t) & g_{ij}(t) \\
g_{ji}(t) & f_{j}(t) \\
\end{array} \right] \geq 0\Bigg\}\nonumber
\end{equation}
\end{lemma}

{In the following, some new weighted integral inequalities are derived by refining those established in  \cite{LiuX} and \cite{Vong}.  Let $p_i(u)$ for $i=0,1,2,3$ be some scalar functions  on $[a,b]$ and the weight function $w(u)$ is large than zero.
Considering a product between two functions as follow}
$$\langle p_i,p_j \rangle=\int_{a}^bp_i(u)p_j(u) w(u)du$$  and  functions $\{p_0,p_1,p_2,p_3\}$ satisfying the ``orthogonal"    properties as follow:
\begin{equation}\label{ortho0}
\int_{a}^{b}  p_0(u)p_i(u)w(u)du=0\,\,( i=1,2,3),\,\, \int_{a}^{b} p_1(u)p_2(u)w(u)du=0, \,\,\int_{a}^{b} p_2(u)p_3(u)w(u)du=0.
\end{equation}
In particular, we take $p_0(u)\equiv1$. The main estimate in this paper read as:
\begin{lemma} \label{Lemma2}
	For a  matrix $R\in\mathbb{S}_n^{+}$, we have
%
{\setlength\arraycolsep{2pt}
\begin{eqnarray} \label{a10}
\int_{a}^{b} \phi^{T}(u)R\phi(u)w(u)du & \geq & \frac{1}{q_0}F_0^{T}RF_0+\frac{1}{q_1}F_{1}^{T}RF_{1}+\frac{1}{q_2}F_{2}^{T}RF_{2}\nonumber\\ & & +\frac{1}{q_3}\biggr[F_{3}-\frac{q_{13}}{q_{1}}F_{1}\biggr]^{T}R\biggr[F_{3}-\frac{q_{13}}{q_{1}}F_{1}\biggr]
\end{eqnarray}}
where
$$F_i=\int_{a}^{b}p_{i}(u)\phi_{i}(u)w(u)du,\quad q_{i}=\int_{a}^{b}p_{i}^{2}(u)w(u)du,\quad i=0,1,2,3, $$
$$F_0=\int_{a}^{b}\phi(u)w(u)du,\quad q_0=\int_{a}^{b}w(u)du,\quad q_{13}=\int_{a}^{b}p_{1}(u)p_{3}(u)w(u)du $$
\end{lemma}
\begin{proof}
 Let
\begin{displaymath}
e(u)=\phi(u)-\frac{F_0}{q_0}-\frac{F_1}{q_1}p_{1}(u)-\frac{F_2}{q_2}p_{2}(u)-p_{3}(u)v
\end{displaymath}
where $v$ is a constant vector in $\mathbb{R}^{n}$. Since $R$ is positive definite, if we take
\begin{displaymath}
v=\frac{F_3}{q_3}-\frac{\int_{a}^{b}p_1(u)p_3(u)w(u)du}{q_1q_3}F_1,
\end{displaymath}
we have
{\setlength\arraycolsep{2pt}
\begin{eqnarray}
&\int_{a}^{b}&e^{T}(u)Re(u)w(u)du  \nonumber\\
&= & \int_{a}^{b}\biggr[\phi(u)-\frac{F_0}{q_0}-\frac{F_1}{q_1}p_{1}(u)-\frac{F_2}{q_2}p_{2}(u)-p_{3}(u)v\biggr]^{T}R \biggr[\phi(u)-\frac{F_0}{q_0}\nonumber\\ & &-\frac{F_1}{q_1}p_{1}(u)-\frac{F_2}{q_2}p_{2}(u)-p_{3}(u)v\biggr]w(u)du\nonumber\\
& = & \int_{a}^{b}\biggr[\phi(u)-\frac{F_0}{q_0}-\frac{F_1}{q_1}p_{1}(u)-\frac{F_2}{q_2}p_{2}(u)\biggr]^{T}R\biggr[\phi(u)-\frac{F_0}{q_0}
-\frac{F_1}{q_1}p_{1}(u)-\frac{F_2}{q_2}p_{2}(u)\biggr]w(u)du\nonumber\\
& &-2\int_{a}^{b}\biggr[\phi(u)p_3(u)-\frac{F_0}{q_0}
p_3(u)-\frac{F_1}{q_1}p_1(u)p_3(u)-\frac{F_2}{q_2}p_2(u)p_3(u)\biggr]^{T}w(u)duRv\nonumber\\
& &
+\int_{a}^{b}p_{3}^{2}(u)w(u)duv^{T}Rv\nonumber\\
& = & \int_{a}^{b}\biggr[\phi(u)-\frac{F_0}{q_0}-\frac{F_1}{q_1}p_{1}(u)-\frac{F_2}{q_2}p_{2}(u)\biggr]^{T}R\biggr[\phi(u)-\frac{F_0}{q_0}
-\frac{F_1}{q_1}p_{1}(u)-\frac{F_2}{q_2}p_{2}(u)\biggr]w(u)du-q_3v^{T}Rv\nonumber\\
& = &
\int_{a}^{b}\phi^T(u)R\phi(u)w(u)du-2\int_{a}^{b}\phi^T(u)R\biggr[\frac{F_0}{q_0}+\frac{F_1}{q_1}p_{1}(u)
+\frac{F_2}{q_2}p_{2}(u)\biggr]w(u)du\nonumber\\
& &+\int_{a}^{b}\biggr[\frac{F_0}{q_0}+\frac{F_1}{q_1}p_{1}(u)+\frac{F_2}{q_2}p_{2}(u)\biggr]^{T}
R\biggr[\frac{F_0}{q_0}+\frac{F_1}{q_1}p_{1}(u)+\frac{F_2}{q_2}p_{2}(u)\biggr]w(u)du-q_3v^{T}Rv\nonumber\\
& = &
\int_{a}^{b}\phi^T(u)R\phi(u)w(u)du-\frac{1}{q_0}F_0^{T}RF_0-\frac{1}{q_1}F_{1}^{T}RF_{1}
-\frac{1}{q_2}F_{2}^{T}RF_{2}-q_{3}v^{T}Rv\geq 0,\nonumber
\end{eqnarray}}
which is equivalent to
{\setlength\arraycolsep{2pt}
\begin{eqnarray}
\int_{a}^{b} \phi^T(u)R\phi(u)w(u)du\geq \frac{1}{q_0}F_0^{T}RF_0+\frac{1}{q_1}F_{1}^{T}RF_{1}+\frac{1}{q_2}F_{2}^{T}RF_{2} +\frac{1}{q_3}\biggr[F_{3}-\frac{q_{13}}{q_{1}}F_{1}\biggr]^{T}R\biggr[F_{3}-\frac{q_{13}}{q_{1}}F_{1}\biggr].\nonumber
\end{eqnarray}}
\end{proof}

Consider the weight function $w(u)=e^{-2k(u-b)}$ and $\phi(u)=e^{2k(u-b)}z(u)$ in \eqref{a10}. We can get the following inequality:
\begin{lemma} \label{Lemma4}
{Consider an integrable function $z:[a,b]\rightarrow \mathbb{R}^n$ and a matrix $R\in\mathbb{S}_n^{+}$.
We have the following inequality:}
\end{lemma}
\begin{eqnarray}\label{a11}
\int_{a}^{b} e^{2k(u-b)}z^{T}(u)Rz(u)du & \geq & \frac{1}{q_0}\Omega_0^{T}R\Omega_0+\frac{1}{q_1}\Omega_{1}^{T}R\Omega_{1}+\frac{1}{q_2}\Omega_{2}^{T}R\Omega_{2}\nonumber\\ &&+\frac{1}{q_3}\biggr[\Omega_{3}-\frac{q_{13}}{q_{1}}\Omega_{1}\biggr]^{T}R\biggr[\Omega_{3}-\frac{q_{13}}{q_{1}}\Omega_{1}\biggr]
\end{eqnarray}
where
{\setlength\arraycolsep{2pt}
	\begin{eqnarray}
	\Omega_0 & = & \int_{a}^{b}z(u)du,\ \ \
	\Omega_1  =  c_{1}\int_{a}^{b}z(u)du+\int_{a}^{b}\int_{s}^{b}z(u)duds,\nonumber\\
	\Omega_2 & = & c_{3}\int_{a}^{b}z(u)du+c_{2}\int_{a}^{b}\int_{s}^{b}z(u)duds+2\int_{a}^{b}\int_{s}^{b}\int_{u}^{b}z(v)dvduds,\nonumber\\
	\Omega_3 & = & \int_{a}^{b}z(u)du+c_4\int_{a}^{\xi}z(u)du,\ \ \
	w = e^{2k(b-a)}, \ \ \
	c_1 = \frac{b-a}{w-1}-\frac{1}{2k},\nonumber\\
	c_2 & = & {-\frac{\frac{(w-1)}{2k^3}-(b-a)^3-\frac{(b-a)^2}{k}-\frac{(b-a)}{2k^2}-\frac{(b-a)^3}{w-1}-\frac{(b-a)^2}{k(w-1)}}{\frac{w-1}{4k^2}-(b-a)^2-\frac{(b-a)^2}{w-1}}}\nonumber\\
	c_3 & = & c_1c_2-(\frac{1}{2k^2}-\frac{(b-a)^2}{w-1}-\frac{(b-a)}{k(w-1)}),\ \ \
	c_4 = -\frac{w-1}{w-e^{-2k(\xi-b)}},\nonumber\\
	q_0 & = & \frac{w-1}{2k},\ \ \
	q_1= {\frac{w-1}{8k^3}-\frac{(b-a)^2}{2k}-\frac{(b-a)^2}{2k(w-1)}},\nonumber\\
	q_2 & = & {\frac{3(w-1)}{4k^5}-\frac{(b-a)^4}{2k}-\frac{(b-a)^3}{k^2}-\frac{3(b-a)^2}{2k^3}- \frac{3(b-a)}{2k^4}-c_2^2q_1-(c_3-c_1c_2)^2q_0.}\nonumber\\
	q_3 & = & {(\frac{w-1}{2k})(\frac{e^{-2k(\xi-b)}-1}{w-e^{-2k(\xi-b)}})},\ \ \
	q_{13}= \frac{(w-1)(\xi-a)e^{-2k(\xi-b)}}{2k(w-e^{-2k(\xi-b)})}-\frac{b-a}{2k}.\nonumber
	\end{eqnarray}}
\begin{proof}
In order to use Lemma \ref{Lemma2}, we first introduce the function $p_3$.  Noting that $p_3$ can be a discontinuous function and it must satisfy $\langle p_3,1\rangle=\langle p_3,p_2\rangle=0$ {and $\langle p_3,p_1\rangle\neq0$}.
Let $\xi\in(a,b)$ be such that $\int_{a}^{\xi}p_2(u)w(u)dt=0$ and denote $p_3=1-\frac{\langle 1,1\rangle}{\langle 1,\chi\rangle}\chi$, where $\chi(u)=\left\{ \begin{array}{ll}
1 & \textrm{if $u\in[a,\xi]$}\\
0 & \textrm{if $u\in(\xi,b]$}
\end{array} \right.$. We can get $\langle p_3,p_2\rangle=\langle 1,p_2\rangle-\frac{\langle 1,1\rangle}{\langle 1,\chi\rangle}\langle\chi,p_2\rangle=0$ and $\langle p_3,1\rangle=\langle 1,1\rangle-\frac{\langle 1,1\rangle}{\langle 1,\chi\rangle} \langle\chi,1\rangle=0$.

	Then we take $p_1(u),~p_2(u)$ as linear and quadratic polynomial:
 $$ p_1(u)=(u-a)+c_1,\ \ \ p_2(u)=(u-a)^2+c_2(u-a)+c_3.$$
	which satisfied $\int_{a}^{b}p_i(u)w(u)du=0\
	\ (i=1,2)$ and $\int_{a}^{b}p_1(u)p_2(u)w(u)du=0$. We can get $c_1=-\frac{\int_{a}^{b}(u-a)w(u)du}{\int_{a}^{b}w(u)du}$, $c_2=-\frac{\int_{a}^{b}(u-a)^2p_1(u)w(u)du}{\int_{a}^{b}p_1(u)p_1(u)w(u)du}$,
	$c_3=c_2c_1-\frac{\int_{a}^{b}(u-a)^2w(u)du}{\int_{a}^{b}w(u)du}$ by  simple calculations.
	
	Denote $c_4=-\frac{\langle 1,1\rangle}{\langle 1,\chi\rangle}$, then  straight computations leads to
	\begin{equation}
	\begin{aligned}
	F_0&= \int_{a}^{b}\phi(u)w(u)du=  \int_{a}^{b}z(u)du=\Omega_0 ,\nonumber\\
	 F_1&= \int_{a}^{b}p_1(u)\phi(u)w(u)du= c_{1}\int_{a}^{b}z(u)du+\int_{a}^{b}\int_{s}^{b}z(u)duds=\Omega_1,\nonumber\\
	 F_2& = \int_{a}^{b}p_2(u)\phi(u)w(u)du= c_{3}\int_{a}^{b}z(u)du+c_{2}\int_{a}^{b}\int_{s}^{b}z(u)duds+2\int_{a}^{b}\int_{s}^{b}\int_{u}^{b}z(v)dvduds=\Omega_2,\nonumber\\
	 F_3& = \int_{a}^{b}p_3(u)\phi(u)w(u)du= \int_{a}^{b}z(u)du+c_4\int_{a}^{\xi}z(u)du=\Omega_3.
	\end{aligned}
	\end{equation}
	
	By Lemma \ref{Lemma2}, the inequality \eqref{a11} holds.

\end{proof}

Particularly, when $w(u)=1$, we have the following lemma.

\begin{lemma} \label{Lemma5}
{\rm\cite{LiuX}}
{Given  {an integrable function $z$} $:[a,b]\rightarrow \mathbb{R}^n$ and a matrix $R\in\mathbb{S}_n^{+}$, one has following:}
{\setlength\arraycolsep{2pt}
\begin{eqnarray}
\int_{a}^{b} z^{T}(u)Rz(u)du & \geq & \frac{1}{b-a}\omega_0^{T}R\omega_0+\frac{3}{b-a}\omega_{1}^{T}R\omega_{1}+\frac{5}{b-a}\omega_{2}^{T}R\omega_{2}\nonumber\\ & & +\frac{1}{b-a}\biggr[\omega_{3}-\frac{3}{2}\omega_{1}\biggr]^{T}R\biggr[\omega_{3}-\frac{3}{2}\omega_{1}\biggr]
\end{eqnarray}}
where
{\setlength\arraycolsep{2pt}
\begin{eqnarray}
\omega_0 & = & \int_{a}^{b}z(u)du,\nonumber\\
\omega_1 & = & \int_{a}^{b}z(u)du-\frac{2}{b-a}\int_{a}^{b}\int_{s}^{b}z(u)duds,\nonumber\\
\omega_2 & = & \int_{a}^{b}z(u)du-\frac{6}{b-a}\int_{a}^{b}\int_{s}^{b}z(u)duds+\frac{12}{(b-a)^2}\int_{a}^{b}\int_{s}^{b}\int_{v}^{b}z(u)dudvds,\nonumber\\
\omega_3 & = & \int_{a}^{\frac{a+b}{2}}z(u)du-\int_{\frac{a+b}{2}}^{b}z(u)du.\nonumber
\end{eqnarray}}
\end{lemma}

\begin{remark}
	{By extending the integral inequality based on the auxiliary function in \cite{LiuX}, we propose a new weighted integral inequality in Lemma \ref{Lemma4}. Our main goal is to derive an  improved and less conservative criterion  for stability analysis of  time-delay neural networks.} As a special case, it can be found that when $w(u)\equiv1$, inequality in  Lemma \ref{Lemma4}  reduces to the inequality in Lemma \ref{Lemma5}.
\end{remark}

\begin{lemma} \label{Lemma6}
{\rm\cite{ParkP}} Given a matrix $R>0$,  for all continuous differentiable functions $x:[a,b]\rightarrow \mathbb{R}^{n}$,
one has the following inequalities:
{\setlength\arraycolsep{2pt}
\begin{eqnarray}
-\int^{b}_{a}\int^{b}_{s}\dot{x}^{T}(u)R\dot{x}(u)duds & \leq & -2\Omega_5^{T}R\Omega_5-4\Omega_6^{T}R\Omega_6,\nonumber\\
-\int^{b}_{a}\int^{s}_{a}\dot{x}^{T}(u)R\dot{x}(u)duds & \leq & -2\Omega_7^{T}R\Omega_7-4\Omega_8^{T}R\Omega_8.\nonumber
\end{eqnarray}}
where
{\setlength\arraycolsep{2pt}
\begin{eqnarray}
\Omega_5 & = & x(b)-\frac{1}{b-a}\int^{b}_{a}x(u)du,\nonumber\\
\Omega_6 & = & x(b)+\frac{2}{b-a}\int^{b}_{a}x(u)du-\frac{6}{(b-a)^2}\int^{b}_{a}\int^{b}_{s}x(u)duds,\nonumber\\
\Omega_7 & = & x(a)-\frac{1}{b-a}\int^{b}_{a}x(u)du,\nonumber\\
\Omega_8 & = & x(a)-\frac{4}{b-a}\int^{b}_{a}x(u)du+\frac{6}{(b-a)^2}\int^{b}_{a}\int^{b}_{s}x(u)duds.\nonumber
\end{eqnarray}}
\end{lemma}

\bigskip

\section{Stability analysis}

{In this section, we prove our main result on exponential stability of \eqref{a6}.}

\begin{theorem} \label{Theorem1}
	
{For given positive constants $h$ and $\mu$, system (\ref{a6}) is
	globally exponentially stable with exponential convergence rate $k:0<k<min_{1\leq i\leq n}c_i$,
  and
  positive definite symmetric matrices
  $P\in \mathbb{R}^{3n\times 3n}$, $Q\in \mathbb{R}^{2n\times 2n}$, $U_i\in \mathbb{R}^{n\times n}$, $Z_i\in \mathbb{R}^{n\times n}$, $i=1,2$ $N_j\in \mathbb{R}^{n\times n}$, $M_j\in \mathbb{R}^{n\times n}$, $j=1,2$, positive definite diagonal matrices $D_i=diag\{d_{i1},\ldots,d_{in}\}\in \mathbb{R}^{n\times n}$, $R_i\in \mathbb{R}^{n\times n}$,  $i=1,2$, and any matrices $S\in \mathbb{R}^{3n\times 3n}$ that fulfill the following LMIs:}
	$$\Phi+\Theta_1<0,\quad \Phi+\Theta_2<0, \quad \Gamma>0$$
	where\\
	$\Phi=\Xi_1+\Xi_2+\Xi_3+\Xi_4+\Xi_5+\Psi+\Pi$, \quad
	$\Theta_1=\varphi_1+\varphi_2,\quad \Theta_2=\psi_1+\psi_2, \nonumber\\ e_i=\biggl[\underbrace{0,0,\ldots,\overbrace{I}^i,\ldots,0}_{12}\biggl]^{T}_{12n\times n}, i= 1,2,\ldots,12,$
{	$e_s=[-C,0_{n\times 2n},A,B,0_{n\times 7n}]^T,$}\nonumber\\
	$P=\left[\begin{array}{ccc}P_{11} & P_{12} & P_{13} \\
	\ast & P_{22} & P_{23} \\
	\ast & \ast & P_{33}\end{array} \right],
	\Gamma=\left[\begin{array}{ccc}Z_{11} & S \\
	S & Z_{12}\end{array}\right],
	\Omega=\left[\begin{array}{ccc}Z_{13} & S \\
	S & Z_{13}\end{array}\right],$\nonumber\\
	$Z_{11}=diag\{Z_1+N_1, 3(Z_1+N_1), 5(Z_1+N_1)\},$\nonumber\\
	$Z_{12}=diag\{Z_1+N_2, 3(Z_1+N_2), 5(Z_1+N_2)\},\,Z_{13}=diag\{Z_1,3Z_1,5Z_1\},$\nonumber\\
	$Z_{14}=diag\{\frac{h}{q_0}Z_3,\frac{h}{q_1}Z_3,\frac{h}{q_2}Z_3\},\,N_{14}=diag\{N_1,3N_1,5N_1\},\,N_{15}=diag\{N_2,3N_2,5N_2\},$\nonumber\\
	$\gamma(1)=[(e_1-e_2),\,(e_1+e_2-2e_7),\,(e_1-e_2+6e_7-6e_{10})],$\nonumber\\
	$\gamma(2)=[(e_2-e_3),\,(e_2+e_3-2e_8),\,(e_2-e_3+6e_8-6e_{11})],$\nonumber\\
	$\gamma(3)=[(e_1-e_3),\,((h+c_1)e_1-c_1e_3-he_6),\,((h^2+c_2h+c_3)e_1-c_3e_3-c_2he_6-h^2e_{9})],$\nonumber\\
	$\gamma=[\gamma(1),\,\gamma(2)], \ \zeta(1)=[e_1,\,he_7,\,he_{9}],\quad \zeta(2)=[e_1,\,he_8,\,he_{9}],$\nonumber\\
	$\zeta(3)=[e_s,\,e_1-e_3,\,2(e_1-e_6)],\quad \zeta(4)=[e_1,\,he_6,\,he_{9}],$\nonumber\\
	$\Xi_1=sym\{k\zeta(4)P\zeta^T(4)\,+\,2k[e_4D_1e_1^T\,+\,(e_1L\,-\,e_4)D_2e_1^T]\,+\,e_4D_1e_s^T+\,(e_1L-e_4)D_2e_s^T\},$\nonumber\\
	$\Xi_2=e^{2kh}\{[e_1,\,e_4]Q[e_1,\,e_4]^T\,+\,e_1U_1e_1^T\,+\,e_1U_2e_1^T\}\,-\,(1\,-\,\mu) [e_2,\,e_5]Q[e_2,\,e_5]^T\,-\,e^{2k(h-\xi)}[e_{12}U_2e_{12}^T-e_{12}U_3e_{12}^T]\,-\,[e_3U_1e_3^T\,+\,e_3U_3e_3^T],$\nonumber\\
	$\Xi_3=h^2(e_sZ_1e_s^T\,+\,e_1Z_2e_1^T\,+\,e_sZ_3e_s^T)\,-\,\biggr[\frac{h^3}{q_0}e_6Z_2e_6^T+\,\frac{h^5}{4q_1}(\frac{2c_1}{h}e_6\,+\,e_{9})Z_2(\frac{2c_1}{h}e_6\,+\,e_{9})^T\,$\nonumber\\
	$+\,\gamma(3)Z_{14}\gamma^T(3)+\frac{h}{q_3}\,\big((1-\frac{(h+c_1)q_{13}}{q_1})e_1-(1+c_4-\frac{c_1q_{13}}{q_1})e_3+\frac{q_{13}}{q_1}e_6+c_4e_{12}\big)Z_3\big((1-\frac{(h+c_1)q_{13}}{q_1})e_1-(1+c_4-\frac{c_1q_{13}}{q_1})e_3+\frac{q_{13}}{q_1}e_6+c_4e_{12}\big)^T\biggr],$\nonumber\\
	$\Xi_4=\frac{h^2}{2}e_sN_1e_s^T\,+\,\frac{h^2}{2}e_sN_2e_s^T\,-\,e^{-2kh}\biggr[2(e_1-e_7)N_1(e_1-e_7)^T+\,4(e_1\,+\,2e_7\,-\,3e_{10})N_1(e_1\,+\,2e_7\,-\,3e_{10})^T\,+\,2(e_2-e_8)N_1(e_2-e_8)^T+\,4(e_2\,+\,2e_8\,-\,3e_{11})N_1(e_2\,+\,2e_8\,-\,3e_{11})^T\,+\,2(e_2-e_7)N_2(e_2-e_7)^T+\,4(e_2\,-\,4e_7\,+\,3e_{10})N_2(e_2\,-\,4e_7\,+\,3e_{10})^T\,+\,2(e_3-e_8)N_2(e_3-e_8)^T$\nonumber\\
	$+\,4(e_3\,-\,4e_8\,+\,3e_{11})N_2(e_3\,-\,4e_8\,+\,3e_{11})^T\biggr],$\nonumber\\
	$\Xi_5=\frac{\mu}{h}e_1(M_1{-M_2})e_1^T,$
	$\Psi=-e^{-2kh}\gamma\Omega\gamma^T,$\nonumber\\
	$\Pi=sym(e_1LR_1e_4^T\,-\,e_4R_1e_4^T\,+\,e_2LR_2e_5^T\,-\,e_5R_2e_5^T),$\nonumber\\
	$\varphi_1=sym(\zeta(1)P\zeta^T(3)),\quad\varphi_2=sym(ke_1M_1e_1^T\,+\,e_1M_1e_s^T),$\nonumber\\
	$\psi_1=sym(\zeta(2)P\zeta^T(3)),\quad\psi_2=sym(ke_1M_2e_1^T\,+\,e_1M_2e_s^T),$\nonumber\\
	$\alpha=\frac{h(t)}{h},\quad\beta=\frac{h-h(t)}{h},\quad L=diag\{L_1,\ldots,L_n\}.$\nonumber\\
\end{theorem}

\begin{proof}
	Consider the following LKF
	$$V(x(t))=\sum_{i=1}^5V_i(x(t))$$
	where
	\begin{equation}
	\begin{array}{l}
	V_1(x(t))=e^{2kt}\alpha^T(t)P\alpha(t)+2\sum^n_{i=1}e^{2kt}d_{1i}\int_{0}^{z_i}f_i(s)ds
	+2\sum^n_{i=1}e^{2kt}d_{2i}\int_{0}^{z_i}(L_is-f_i(s))ds,\nonumber\\
	V_2(x(t))=e^{2kh}\biggr\{\int_{t-h(t)}^{t}e^{2ks}\varepsilon^T(s)Q\varepsilon(s)ds+\int_{t-h}^{t}e^{2ks}z^T(s)U_1z(s)ds+
	\int_{t-\xi}^{t}e^{2ks}z^T(s)U_2z(s)ds\nonumber\\
	~~~~~~~~~~~~~~+\int_{t-h}^{t-\xi}e^{2ks}z^T(s)U_3z(s)ds\biggr\},\nonumber\\
	V_3(x(t))=h\biggr\{\int_{-h}^{0}\int_{t+u}^{t}e^{2ks}\dot{z}(s)^TZ_1\dot{z}(s)dsdu+\int_{-h}^{0}\int_{t+u}^{t}e^{2ks}z(s)^TZ_2z(s)dsdu\nonumber\\
	~~~~~~~~~~~~~~+\int_{-h}^{0}\int_{t+u}^{t}e^{2ks}\dot{z}(s)^TZ_3\dot{z}(s)dsdu\biggr\},\nonumber\\
	V_4(x(t))= \int_{-h}^{0}\int_{v}^{0}\int_{t+u}^{t}e^{2ks}\dot{z}^T(s)N_1\dot{z}(s)dsdudv+ \int_{-h}^{0}\int_{-h}^{v}\int_{t+u}^{t}e^{2ks}\dot{z}^T(s)N_2\dot{z}(s)dsdudv,\nonumber\\
	V_5(x(t))= \frac{h(t)}{h}e^{2kt}z^T(t)M_1z(t)+\frac{h-h(t)}{h}e^{2kt}z^T(t)M_2z(t).\nonumber
	\end{array}
	\end{equation}
	Let
	$\eta^T(t) = [z^T(t),\,z^T(t-h(t)),\,z^T(t-h),\,f^T(z(t)),\,f^T(z(t-h(t)),
	\ \frac{1}{h}\int_{t-h}^{t}z^T(s)ds,\\\,\frac{1}{h(t)}\int_{t-h(t)}^{t}z^T(s)ds,\,\frac{1}{h-h(t)}\int_{t-h}^{t-h(t)}z^T(s)ds,\,
	\frac{2}{h^2}\int_{-h}^{0}\int_{t+u}^{t}z^T(s)dsdu,\,\frac{2}{h^2(t)}\int_{-h(t)}^{0}\int_{t+u}^{t}z^T(s)dsdu,\,\nonumber\\
	\frac{2}{(h-h(t))^2}\int_{-h}^{-h(t)}\int_{t+u}^{t-h(t)}z^T(s)dsdu,\,z^T(t-\xi)],$ $
	\alpha^T(t)  =  [z^T(t),\,\int_{t-h}^{t}z^T(s)ds,\,\frac{2}{h}\int_{-h}^{0}\int_{t+u}^{t}z^T(s)dsdu],$ $
	\varepsilon^T(t)  = [z^T(t),\,f^T(z(t))].
	$
	In the following, we estimate time derivative of $V_i(z(t)),\quad i=1,2,3,$ along trajectories of (\ref{a1}).
    {The following of three estimates are similar to those in \cite{LiuX}
    but we still give some critical steps for the completeness of our presentation:}
    \begin{equation}
	\begin{aligned}
	\dot{V}_1(z(t))\leq&
	e^{2kt}\eta^T(t)[\Xi_1+\alpha\varphi_1+\beta\psi_1]\eta(t) \nonumber\\
	\dot{V}_2(z(t))  =&  e^{2kh}\bigg[e^{2kt}\epsilon^T(t)Q\epsilon(t)-e^{2k(t-h(t))}(1-{h'(t)})\epsilon^T(t-h(t))Q\epsilon(t-h(t))
	+e^{2kt}z^T(t)U_1z(t)\nonumber\\ &-e^{2k(t-h)}z^T(t-h)U_1z(t-h)+e^{2kt}z^T(t)U_2z(t)
	-e^{2k(t-\xi)}z^T(t-\xi)U_2z(t-\xi)\nonumber\\ &+e^{2k(t-\xi)}z^T(t-\xi)U_3z(t-\xi)
	-e^{2k(t-h)}z^T(t-h)U_3z(t-h)\bigg]
	\nonumber\\  \leq  &e^{2kt}\eta^T(t)\bigg\{e^{2kh}[e_1,e_4]Q[e_1,e_4]^T-(1-\mu)[e_2,e_5]Q[e_2,e_5]^T
	+e^{2kh}e_1U_1e_1^T-e_3U_1e_3^T\nonumber\\ &+e^{2kh}e_1U_2e_1^T-e^{2k(h-\xi)}e_{12}U_2e_{12}^T
	+e^{2k(h-\xi)}e_{12}U_3e_{12}^T-e_3U_3e_3^T\bigg\}
	\nonumber\\
	=&  e^{2kt}\eta^T(t)\Xi_2\eta(t) \nonumber\\
\end{aligned}
	\end{equation}
\begin{eqnarray}\nonumber
		\dot{V}_4(z(t)) & = & \frac{h^2}{2}e^{2kt}\dot{z}^T(t)(N_1+N_2)\dot{z}(t)-\int_{-h}^{0}\int_{t+u}^{t}e^{2ks}\dot{z}^T(s)N_1\dot{z}(s)dsdu\nonumber\\
		& & -\int_{-h}^{0}\int_{t-h}^{t+u}e^{2ks}\dot{z}^T(s)N_2\dot{z}(s)dsdu\nonumber\\
		& \leq & \frac{h^2}{2}e^{2kt}\dot{z}^T(t)(N_1+N_2)\dot{z}(t)-e^{2k(t-h)}\int_{-h}^{0}\int_{t+u}^{t}\dot{z}^T(s)N_1\dot{z}(s)dsdu\nonumber\\
		& & -e^{2k(t-h)}\int_{-h}^{0}\int_{t-h}^{t+u}\dot{z}^T(s)N_2\dot{z}(s)dsdu\nonumber\\
		&\leq & e^{2kt}\eta^T(t)\Bigg\{\Xi_4-e^{-2kh}\bigg[\Big(\frac1\alpha-1\Big)\gamma(1)N_{14}\gamma^T(1)+\Big(\frac1\beta-1\Big)\gamma(2)N_{15}\gamma^T(2)\bigg]\Bigg\}\nonumber.
		\end{eqnarray}
We use our novel inequalities in Lemma \ref{Lemma4} to estimate $\dot{V}_3$. To this end, we write
    \begin{equation}
    \begin{aligned}
	\dot{V}_3(z(t))  =&  e^{2kt}\bigg[h^2\dot{z}^T(t)(Z_1+Z_3)\dot{z}(t)+h^2z^T(t)Z_2z(t)
	-h\int_{t-h}^{t}e^{2k(s-t)}z^T(s)Z_2z(s)ds\nonumber\\ &-h\int_{t-h}^{t}e^{2k(s-t)}\dot{z}^T(s)Z_3\dot{z}(s)ds\bigg]
	-h\int_{t-h}^{t}e^{2ks}\dot{z}^T(s)Z_1\dot{z}(s)ds\nonumber
	\end{aligned}
	\end{equation}

	Similar to  \cite{LiuX}, by using Lemma \ref{Lemma5},
	\begin{equation}
	\begin{aligned}
	-h\int_{t-h}^{t}e^{2ks}\dot{z}^T(s)Z_1\dot{z}(s)ds
	\leq e^{2k(t-h)}\eta^T(t)\bigg\{\frac{1}{\alpha}\gamma(1)Z_{13}\gamma^T(1)+\frac{1}{\beta}\gamma(2)Z_{13}\gamma^T(2)\bigg\}\eta(t).\nonumber
	\end{aligned}
	\end{equation}
	We next make use of {\eqref{a11}} to get that
	\begin{equation}
	\begin{aligned}
	-h\int_{t-h}^{t}&e^{2k(s-t)}z^T(s)Z_2z(s)ds \nonumber\\
	& \leq -\frac{h}{q_0}\Bigg[\int^t_{t-h}z^T(s)ds\Bigg]Z_2\Bigg[\int^t_{t-h}z(s)ds\Bigg] -\frac{h}{q_1}[c_1\int^t_{t-h}z(s)ds+\int^0_{-h}\int^t_{t+u}z(s)dsdu]^T\nonumber\\
	&~~\times Z_2 [c_1\int^t_{t-h}z(s)ds+\int^0_{-h}\int^t_{t+u}z(s)dsdu]\nonumber\\
	& = -\eta^T(t)\bigg\{\frac{h^3}{q_0}e_6Z_2e_6^T+\frac{h^5}{4q_1}(\frac{2c_1}{h}e_6+e_{9})Z_2(\frac{2c_1}{h}e_6+e_{9})^T\bigg\}\eta(t)\nonumber
	\end{aligned}
	\end{equation}
\begin{equation}
	\begin{aligned}
	-h\int_{t-h}^{t}&e^{2k(s-t)}\dot{z}^T(s)Z_3\dot{z}(s)ds\nonumber\\
	&  \leq
	-\frac{h}{q_0}[z(t)-z(t-h)]^TZ_3[z(t)-z(t-h)] -\frac{h}{q_1}\bigg[(c_1+h)z(t)-c_1z(t-h)-\int^t_{t-h}z(s)ds\bigg]^T\nonumber\\
	& ~~\times Z_3 \bigg[(c_1+h)z(t)-c_1z(t-h)-\int^t_{t-h}z(s)ds\bigg]\nonumber\\
	&~~ -\frac{h}{q_2}\bigg[(h^2+c_2h+c_3)z(t)-c_3z(t-h)-c_2\int^t_{t-h}z(s)ds-2\int^0_{-h}\int^t_{t+u}z(s)dsdu\bigg]^TZ_3\nonumber\\
	&~~\times \bigg[(h^2+c_2h+c_3)z(t)-c_3z(t-h)-c_2\int^t_{t-h}z(s)ds-2\int^0_{-h}\int^t_{t+u}z(s)dsdu\bigg]\nonumber\\
	&~~ -\frac{h}{q_3}\bigg[(1-\frac{(h+c_1)q_{13}}{q_1})z(t)-(1+c_4-\frac{c_1q_{13}}{q_1})z(t-h) +c_4z(t-\xi)+\frac{q_{13}}{q_1}\int^t_{t-h}z(s)ds\bigg]^T\nonumber\\
	&~~\times Z_3\bigg[(1-\frac{(h+c_1)q_{13}}{q_1})z(t) -(1+c_4-\frac{c_1q_{13}}{q_1})z(t-h)+c_4z(t-\xi)+\frac{q_{13}}{q_1}\int^t_{t-h}z(s)ds\bigg]\nonumber\\
	& = \eta^T(t)\Bigg\{\gamma(3)Z_{14}\gamma^T(3)+\bigg[(1-\frac{(h+c_1)q_{13}}{q_1})e_1-(1+c_4-\frac{c_1q_{13}}{q_1})e_3 +\frac{q_{13}}{q_1}e_6+c_4e_{12}\bigg]^T\nonumber\\
	&~~\times Z_3\bigg[(1-\frac{(h+c_1)q_{13}}{q_1})e_1-(1+c_4-\frac{c_1q_{13}}{q_1})e_3 +\frac{q_{13}}{q_1}e_6+c_4e_{12}\bigg]\Bigg\}\eta(t).\nonumber
	\end{aligned}
	\end{equation}
	Consequently
	\setlength\arraycolsep{2pt}
		\begin{eqnarray}\nonumber
		\dot{V}_3(z(t)) & \leq & e^{2kt}\eta^T(t)\Bigg\{\Xi_3-e^{-2kh}\bigg[\frac{1}{\alpha}\gamma(1)Z_{13}\gamma^T(1)
		+\frac{1}{\beta}\gamma(2)Z_{13}\gamma^T(2)\bigg]\Bigg\}\eta(t).\nonumber
        \end{eqnarray}
Noting that  $f'$ may have sign changes, different from \cite{LiuX}, we estimate the time derivative of $V_5(z(t))$ as
	\begin{eqnarray}
	\dot{V}_5(z(t))& = & e^{2kt}\alpha\big[2kz^T(t)M_1z(t)+2z^T(t)M_1\dot{z}(t)\big]+\frac{\dot{h}(t)}{h}e^{2kt}z^T(t)M_1z(t)\nonumber\\
	& &
	+e^{2kt}\beta\big[2kz^T(t)M_2z(t)+2z^T(t)M_2\dot{z}(t)\big]-\frac{\dot{h}(t)}{h}e^{2kt}z^T(t)M_2z(t)\nonumber\\
	& \leq & e^{2kt}\eta^T(t)\big\{\alpha sym(ke_1M_1e_1^T+e_1M_1e_s^T)+\beta sym(ke_1M_2e_1^T+e_1M_2e_s^T)+\frac{\mu}{h}e_1(M_1{-M_2})e_1^T\big\}\eta(t)\nonumber\\
	& = &  e^{2kt}\eta^T(t)\{\Xi_5+\alpha\varphi_2+\beta\psi_2\}\eta(t).\nonumber
	\end{eqnarray}
{By considering the assumptions {\eqref{a4}} at $z(t)$ and $z(t-h(t))$,}
for any diagonal matrices $\,\,R_1>0,\,\,R_2>0$, we have
	{\setlength\arraycolsep{2pt}
		\begin{eqnarray}\label{R-inequality}
		0 & \leq & 2e^{2kt}[z^T(t)LR_1{f}(z(t))-f^T(z(t))R_1f(z(t))
\nonumber\\		& & +z^T(t-h(t)) LR_2{f}(z(t-h(t)))-f^T(z(t-h(t)))R_2f(z(t-h(t)))]\nonumber\\
		& = & e^{2kt}\eta^T(t)\Pi\eta(t).
		\end{eqnarray}}
	Using {lemma \ref{Lemma1}}, we have
	{\setlength\arraycolsep{2pt}
		\begin{eqnarray}
		& -e^{-2kh}\eta^T(t)\Bigg\{\frac{1}{\alpha}\gamma(1)Z_{13}\gamma^T(1)+\frac{1}{\beta}\gamma(2)Z_{13}\gamma^T(2)
		& \frac{1}{\alpha}\gamma(1)N_{14}\gamma^T(1)+\frac{1}{\beta}\gamma(2)N_{15}\gamma^T(2)\nonumber\\
		& -\gamma(1)N_{14}\gamma^T(1)-\gamma(2)N_{15}\gamma^T(2)\Bigg\}\eta(t)\nonumber\\
		&\leq  \eta^T(t)\big\{-e^{-2kh}\gamma\Omega\gamma^T\big\}\eta(t)=\eta^T(t)\Psi\eta(t).\nonumber
		\end{eqnarray}}
	Hence,$\,\,\dot{V}(z(t))\leq e^{2kt}\eta^T(t)\{\Phi+\alpha\Theta_1+\beta\Theta_2\}\eta(t).$
	Since $\Phi+\Theta_1<0,\,\Phi+\Theta_2<0$ and $\alpha+\beta=1$, we can get $\Phi+\alpha\Theta_1+\beta\Theta_2<0$,
	then for any $\eta(t)\neq0$ we have $\dot{V}(z(t))<0$ .\\
	One can easily check that,
	$$V(z(0))\leq\Lambda\|\phi\|^2,$$
	and\\
	{\setlength\arraycolsep{2pt}
		\begin{eqnarray}
		\Lambda & = & \lambda_{max}(P)(1+2h^2)+2\lambda_{max}(D_1L)+2\lambda_{max}(D_2L)+he^{2kh}\lambda_{max}(Q)\nonumber\\
		& & \times[1+\lambda_{max}(L^2)]+he^{2kh}(\lambda_{max}(U_1)+\lambda_{max}(U_2)+\lambda_{max}(U_3))\nonumber\\
		& & +\Bigg[\frac{h^3}{2}\lambda_{max}(Z_1)+\frac{h^3}{2}\lambda_{max}(Z_3)+\frac{h^3}{6}\lambda_{max}(N_1)+\frac{h^3}{2}\lambda_{max}(N_2)\Bigg]\nonumber\\
		& &
		\times \big[\lambda_{max}(C^TC)+\lambda_{max}(A^TA)\lambda_{max}(L^2)+\lambda_{max}(B^TB)\lambda_{max}(L^2)\big]\nonumber\\
		& &
		+h\lambda_{max}(M_1+M_2)+\frac{h^3}{2}\lambda_{max}(Z_2).\nonumber
		\end{eqnarray}}
	At the same time, we have
	$$V(z(t))\geq e^{2kt}\alpha^T(t)P\alpha(t)\geq e^{2kt}\lambda_{max}(P)\|\alpha(t)\|^2\geq e^{2kt}\lambda_{max}(P)\|z(t)\|^2.$$
	Therefore,
	$$\|z(t)\|\leq\sqrt{\frac{\Lambda}{\lambda_{max}(P)}}\|\phi\|e^{-kt},$$
	{which completes the proof.}
\end{proof}
\begin{remark}
 {In \cite{LiuX}, when analysing $V_5$,  it was assumed that  $\dot{h}(t)\ge0$.
  We do not impose this restriction in our proof.
  Furthermore, in the inequality \eqref{R-inequality} for the  activation function,
  we only consider relation between $z(t)$, $f(z(t))$ and $z(t-h(t))$,   $f(z(t-h(t))$,
  but remove the relation between $f(z(t-h)$, $z(t-h)$ which was included in the analysis of  \cite{LiuX}.
  Numerical simulation shows that this will not affect the performance of the stability criterion
  while reducing its number of decision variables.}

\end{remark}

\bigskip
\section{Numerical experiments}
{We now test three examples along with their simulations to show the advantages of the obtained results.}

{\bf Example 1}  {\rm\cite{WuM,JiM1,JiM2,LiuX}} Consider the delayed neural network \eqref{a6} with:
\setlength{\abovedisplayskip}{10pt}
\begin{equation}\nonumber
\begin{array}{l}
A=\left[\begin{matrix}
-1& 0.5 \\
0.5& -1
\end{matrix}\right],\ \
B=\left[
\begin{matrix}
-0.5& 0.5 \\
0.5 & 0.5
\end{matrix}
\right],\ \ C=\diag\{2,3.5\},
\ \ L_1=1,\ \ L_2=1.
\end{array}
\end{equation}


{For various $\mu$ and $h=1$, the maximal value for  allowable exponential convergence rate $k$
of the system are recorded in Table \ref{table1}.
From the table, one can notice that our criterion is more effective than the those in \cite{WuM,JiM1,JiM2,LiuX}.}


\begin{table}[H]
	\caption{Allowable values of $k$ for different $\mu$ and $h=1$ (Example 1).}\label{table1}
	\setlength{\tabcolsep}{7mm}
	\centering
	\begin{tabular}{lcccccc}
		\hline
		$\mu$ &0   &  0.8   &0.9 & NoDVs  \\ \hline
		\cite{WuM} &1.15& 0.8643& 0.8344 & $3n^2+12n$  \\
		\cite{JiM1}&1.1540& 0.8696& 0.8354 & $13n^2+6n$  \\
		\cite{JiM2} &1.1544& 0.8784& 0.8484&  $7n^2+8n$ \\
		\cite{LiuX} &1.2147 &0.9382& 0.9104 & ${20.5n^2 + 12.5n}$ \\
Theorem \ref{Theorem1} &1.2477 &1.0299 &1.0115 &${20.5n^2 + 11.5n}$  \\  \hline
	\end{tabular}
\end{table}


{\bf Example 2}   The delayed neural network \eqref{a6} having the following matrices were studied in {\cite{WuM,ZhengC,JiM1,JiM2,LiuX}}:
\setlength{\abovedisplayskip}{10pt}
\begin{equation}\nonumber
\begin{array}{l}
A=\left[\begin{matrix}
-0.0373& 0.4852& -0.3351&0.2336 \\
-1.6033&0.5988&-0.3224&1.2352\\
0.3394& -0.0860&-0.3824& -0.5785\\
-0.1311&0.3253&-0.9534&-0.5015
\end{matrix}\right],\ \
B=\left[
\begin{matrix}
0.8674&-1.2405&-0.5325&-0.0220\\
0.0474&-0.9164&0.0360&0.9816\\
1.8495&2.6117&-0.3788&0.0824\\
-2.0413&0.5179&1.1734&-0.2775
\end{matrix}
\right],\\

\ \ C=\diag\{1.2769, 0.6231, 0.9230, 0.4480\},\\
\ \ L_1=0.1137,\ \ L_2=0.1279,\ \ L_3=0.7994,\ \ L_4=0.2368.
\end{array}
\end{equation}

{For this example, as in \cite{LiuX}, we make a comparison with the methods proposed in \cite{WuM,ZhengC,JiM1,JiM2,LiuX} by taking $k=10^{-6}$. For different  $\mu$, the maximal upper bounds of $h(t)$ with corresponding NoDVs are showed in  Table \ref{table2}. From the reuslt, we can see the improvement  of our method.}


{Fig. \ref{figure2} depicts the trajectory of the delayed system \eqref{a6} when $z(0)=[-1,-0.5,0.5,1]^T,$ $ h(t)=2.8674+0.8sin(t),$ $f(z(t))=[0.1137tanh(z_1(t)),0.1279tanh(z_2(t)), 0.7994tanh(z_3(t)),\\0.2368tanh(z_4(t))].$ }


\begin{table}[H]
	\caption{Allowable $h$ for various $\mu$  (Example 2).}\label{table2}
	\setlength{\tabcolsep}{6mm}
	\centering
	\begin{tabular}{lcccccc}
		\hline
		$\mu$ &0.5   &  0.8   &0.9 & NoDVs  \\ \hline
		\cite{WuM} &2.5379 &2.1766& 2.0853& $3n^2+12n$  \\
		\cite{ZhengC}&2.6711& 2.2977 &2.1783& $4.5n^2 + 17.5n$  \\
		\cite{JiM1}&3.4311& 2.5710& 2.4147 & $13n^2+6n$  \\
		\cite{JiM2}&3.6954& 2.7711 &2.5795&  $7n^2+8n$ \\
		Theorem 3.1\cite{LiuX}$(k=10^{-6})$&3.8709& 3.3442 &3.1291&  ${20.5n^2 + 12.5n}$ \\
Theorem of \ref{Theorem1}$(k=10^{-6})$ &4.2050 &3.6674 &3.5170 &${20.5n^2 + 11.5n}$  \\  \hline
	\end{tabular}
	
\end{table}
\begin{figure}[htb!]
	\setlength{\unitlength}{1cm} 
	\begin{center}
		\resizebox{!}{8cm}{\includegraphics{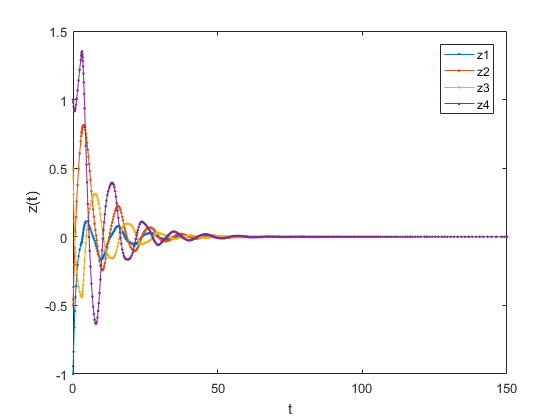}}
		\vspace{-0.1cm}
		\caption{  Trajectory  of Example 2.}\label{figure2}
	\end{center}
\end{figure}

{\bf Example 3} {\rm\cite{ChenW,HeY1,HuaC,HeYLiu,LiuX}} Consider the delayed neural network \eqref{a6} with:
\setlength{\abovedisplayskip}{10pt}
\begin{equation}\nonumber
\begin{array}{l}
A=\left[\begin{matrix}
1& 1 \\
-1& -1
\end{matrix}\right],\ \
B=\left[
\begin{matrix}
0.88& 1 \\
1 & 1
\end{matrix}
\right],\ \ C=\diag\{2,2\},
\ \ L_1=0.4,\ \ L_2=0.8.
\end{array}
\end{equation}

{This example was studied in \cite{ChenW}. We list the maximal delay bounds of $h(t)$ with different $\mu$ and fixed $k=10^{-6}$ in Table \ref{table3}. It is obvious that the results obtained by Theorem \ref{Theorem1} is better than those in \cite{ChenW,HeY1,HuaC,HeYLiu,LiuX}. The improvement show the effectiveness and superiority of our method .}


{Set $z(0)=[-1,1]^T$. The trajectory  of the delayed system \eqref{a6} with $h(t)=6.3039+0.77sin(t),\ f(z(t))=[0.4tanh(z_1(t)),0.8tanh(z_2(t))]$ is depicted in Fig. \ref{figure3}.
}

\begin{table}[H]
	\caption{Allowable $h$ for different $\mu$  (Example 3).}\label{table3}
	\setlength{\tabcolsep}{6mm}
	\centering
	\begin{tabular}{lcccccc}
		\hline
		$\mu$ &0.77   &  0.80   &0.90 & NoDVs  \\ \hline
		\cite{HeY1} &2.3368& 1.2281& 0.8636& $3.5n^2 + 15.5n$  \\
		\cite{HuaC}&2.3368& 1.2281& 0.8636& $14.5n^2 + 7.5n$  \\
		\cite{HeYLiu}&3.2681& 1.6831 &1.1493& $2.5n^2 + 15.5n $ \\
	Theorem 2 with $N=1$\cite{ChenW}&3.4373& 1.8496& 1.0904& $22n^2+8n$ \\
	Theorem 2 with $N=2$\cite{ChenW}&3.5423 &1.9149& 1.1786& $23.5n^2 + 9.5n$ \\
		Theorem 3.1\cite{LiuX}$(k=10^{-6})$&5.8372& 3.3805& 2.1714&  ${20.5n^2 + 12.5n}$ \\
			Theorem of \ref{Theorem1}$(k=10^{-6})$ &7.0739 &3.5641 &2.2092 &${20.5n^2 + 11.5n}$   \\  \hline
	\end{tabular}
	
\end{table}

\begin{figure}[htb!]
\setlength{\unitlength}{1cm} 
\begin{center}
	\resizebox{!}{8cm}{\includegraphics{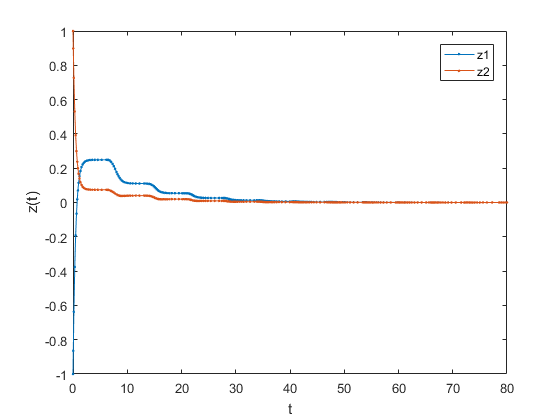}}
	\vspace{-0.1cm}
	\caption{ Trajectory of Example 3.}\label{figure3}
\end{center}
\end{figure}

\section{Conclusion}
Exponential stability for a kind of neural networks having time-varying delay is studied by extend the auxiliary function-based integral inequality with weight functions.
This weighted integral  inequality is used to analyze a Lyapunov-Krasovskii function to obtain a sharpened criterion for exponential stability.
Furthermore, when studying the Lyapunov-Krasovskii function, we find that some decision variables introduced previously can be removed without affecting performance of the proposed criterion.
{Several examples have been tested to demonstrate the advantages of the new criterion.}

\end{document}